\documentclass[a4paper]{amsart}

\title{On the Shapes of Bilateral Gamma Densities}
\author{Uwe K{\"u}chler \and Stefan Tappe}

\usepackage{amscd}
\usepackage{amsmath}
\usepackage{amssymb}
\usepackage{amsthm}
\usepackage{bbm}
\usepackage{eucal}

\newif\ifpdf
\ifx\pdfoutput\undefined
   \pdffalse        
\else
   \pdfoutput=1     
   \pdftrue
\fi

\ifpdf
   \usepackage[pdftex]{graphicx}
\else
   \usepackage{graphicx}
\fi

\numberwithin{equation}{section}
\swapnumbers

\newtheorem{satz}{Satz}[section]

\newtheorem{theorem}[satz]{Theorem}
\newtheorem{proposition}[satz]{Proposition}

\newcommand{\abs}[1]{\lvert #1 \rvert}

\begin{document}

\maketitle\thispagestyle{empty}

\begin{abstract}
We investigate the four parameter family of bilateral Gamma distributions. The goal of this paper is to provide a thorough treatment of the shapes of their densities, which is of importance for assessing their fitting properties to sets of real data. This includes appropriate representations of the densities, analyzing their smoothness, unimodality and asymptotic behaviour. \bigskip

\textbf{Key Words:} bilateral Gamma distributions, selfdecomposability, unimodality, asymptotic behaviour, density shapes
\end{abstract}

\keywords{60G51, 60E07}

\section{Introduction}

In many fields of applications it is important to find appropriate
classes of distributions for fitting observed data. For this issue,
normal distributions often provide only a poor fit. Specific
examples are given by the logarithmic returns of stock prices,
because their empirical densities typically possess heavier tails
and much higher located modes than normal distributions.

Thus, several authors have looked for other appropriate classes of distributions. We mention the generalized hyperbolic distributions \cite{Barndorff-Nielsen} and their subclasses, which have been applied to finance in \cite{Eberlein-Keller}, the Variance Gamma distributions \cite{Madan} and CGMY-distributions \cite{CGMY}.

Recently, another family of distributions was proposed in
\cite{Kuechler-Tappe}: Bilateral Gamma distributions. In the
mentioned article, bilateral Gamma distributions are fitted to
observed stock prices and compared to other classes of distributions
considered in the literature.

In order to provide a general overview about their fitting properties -- also in view of other applications than finance -- we present a thorough treatment of the shapes of their densities. After recalling the basic properties of bilateral Gamma distributions in Section \ref{sec-basic}, we provide suitable representations of the densities in Section \ref{sec-dens}, which we can use in order to obtain density plots with a computer program. Afterwards, the investigation of the shapes of bilateral Gamma distributions starts: Section \ref{sec-smoothness} concerns the smoothness of the densities, Section \ref{sec-unimodality} the unimodality and Section \ref{sec-asymptotic} is devoted to the asymptotic behaviour of the densities near zero and for $x \rightarrow \pm \infty$. In Section \ref{sec-shapes} we characterize typical shapes of the densities and draw implications concerning the fitting properties of bilateral Gamma distributions.

\section{Bilateral Gamma distributions}\label{sec-basic}

In this section, we define bilateral Gamma distributions and review some of their properties. For details and more informations, we refer to \cite{Kuechler-Tappe}.

A {\em bilateral Gamma distribution} with parameters $\alpha^+,
\lambda^+ ,\alpha^-, \lambda^- > 0$ is defined as the distribution
of $X-Y$, where $X$ and $Y$ are independent, $X \sim
\Gamma(\alpha^+,\lambda^+)$ and $Y \sim \Gamma(\alpha^-,\lambda^-)$.

The characteristic function of a bilateral Gamma distribution is
\begin{align}\label{cf-bilateral}
\varphi(z) = \left( \frac{\lambda^+}{\lambda^+ - iz}
\right)^{\alpha^+} \left( \frac{\lambda^-}{\lambda^- + iz}
\right)^{\alpha^-}, \quad z \in \mathbb{R}
\end{align}
where the powers stem from the main branch of the complex logarithm.

If $X$ is bilateral Gamma distributed with parameters $(\alpha^+,\lambda^+;\alpha^-,\lambda^-)$, then for any $c > 0$ the random variable $cX$ has, by (\ref{cf-bilateral}), again a bilateral Gamma distribution with parameters $(\alpha^+,\frac{\lambda^+}{c};\alpha^-,\frac{\lambda^-}{c})$.

Note that, also by (\ref{cf-bilateral}), the sum of two independent bilateral Gamma random variables with parameters $(\alpha_1^+,\lambda^+;\alpha_1^-,\lambda^-)$ and $(\alpha_2^+,\lambda^+;\alpha_2^-,\lambda^-)$ has again a bilateral Gamma distribution with parameters $(\alpha_1^+ + \alpha_2^+,\lambda^+;\alpha_1^- +
\alpha_2^-,\lambda^-)$. In particular, bilateral Gamma distributions are stable under convolution, and they are \textit{infinitely divisible}. It follows from \cite[Ex. 8.10]{Sato} that both, the drift and the Gaussian part in the L\'evy-Khintchine formula (with truncation function $h = 0$), are equal to zero, and that the L\'evy measure is given by
\begin{align}\label{levy-measure}
F(dx) = \left( \frac{\alpha^+}{x} e^{-\lambda^+ x}
\mathbbm{1}_{(0,\infty)}(x) + \frac{\alpha^-}{|x|} e^{-\lambda^-
|x|} \mathbbm{1}_{(-\infty,0)}(x) \right)dx.
\end{align}
Thus, we can also express the characteristic function $\varphi$ as
\begin{align}\label{char-fkt-self}
\varphi(z) = \exp \left( \int_{\mathbb{R}} \left( e^{izx} - 1 \right) \frac{k(x)}{x} dx \right), \quad z \in \mathbb{R}
\end{align}
where $k : \mathbb{R} \rightarrow \mathbb{R}$ is the function
\begin{align}\label{fkt-k-self}
k(x) = \alpha^+ e^{-\lambda^+ x} \mathbbm{1}_{(0,\infty)}(x) - \alpha^- e^{-\lambda^- |x|} \mathbbm{1}_{(-\infty,0)}(x), \quad x \in \mathbb{R}
\end{align}
which is decreasing on each of $(-\infty,0)$ and $(0,\infty)$.
It is an immediate consequence of \cite[Cor. 15.11]{Sato} that bilateral Gamma distributions are \textit{selfdecomposable}, and hence of class $L$ in the sense of \cite{Sato-Yamazato} and \cite{Sato-Y}.
This is a key property for analyzing their densities, which is exploited in Sections \ref{sec-smoothness}, \ref{sec-unimodality} and \ref{sec-asymptotic}.

Using the characteristic function (\ref{cf-bilateral}), we can specify the following quantities.

\begin{center}
\begin{tabular}{lc}
Mean: & $\frac{\alpha^+}{\lambda^+} - \frac{\alpha^-}{\lambda^-}$, \medskip
\\ Variance: & $\frac{\alpha^+}{(\lambda^+)^2} + \frac{\alpha^-}{(\lambda^-)^2}$, \medskip
\\ Skewness: & $2 \left( \frac{\alpha^+}{(\lambda^+)^3} - \frac{\alpha^-}{(\lambda^-)^3} \right) \Big/ \left(
\frac{\alpha^+}{(\lambda^+)^2} + \frac{\alpha^-}{(\lambda^-)^2} \right)^{3/2}$, \medskip
\\ Kurtosis: & $3 + 6 \left( \frac{\alpha^+}{(\lambda^+)^4} +
\frac{\alpha^-}{(\lambda^-)^4} \right) \Big/ \left(
\frac{\alpha^+}{(\lambda^+)^2} + \frac{\alpha^-}{(\lambda^-)^2} \right)^2$.
\end{tabular}
\end{center}

\section{Representations of the densities}\label{sec-dens}

Bilateral Gamma distributions are absolutely continuous with
respect to the Lebesgue measure, because they are the convolution of two Gamma distributions.
Since the densities satisfy the symmetry relation
\begin{align}\label{symm-rel}
f(x;\alpha^+,\lambda^+,\alpha^-,\lambda^-) = f(-x;\alpha^-,\lambda^-,\alpha^+,\lambda^+), \quad x \in \mathbb{R} \setminus \{ 0 \}
\end{align}
it is sufficient to analyze the density functions on the positive real line.
As the convolution of two Gamma densities, they are for $x \in (0,\infty)$ given by
\begin{align}\label{density-bilateral-conv}
f(x) = \frac{(\lambda^+)^{\alpha^+}
(\lambda^-)^{\alpha^-}}{(\lambda^+ + \lambda^-)^{\alpha^-}
\Gamma(\alpha^+) \Gamma(\alpha^-)} e^{-\lambda^+ x} \int_0^\infty
v^{\alpha^- - 1} \left( x + \frac{v}{\lambda^+ + \lambda^-}
\right)^{\alpha^+ - 1} e^{-v} dv.
\end{align}
We can express the density $f$ by means of the \textit{Whittaker function} $W_{\lambda,\mu}(z)$
\cite[p. 1014]{Gradstein}. According to \cite[p. 1015]{Gradstein}, the Whittaker function has the representation
\begin{align}\label{repr-whittaker}
W_{\lambda,\mu}(z) = \frac{z^{\lambda} e^{-\frac{z}{2}}}{\Gamma(\mu-\lambda+\frac{1}{2})}
\int_0^{\infty} t^{\mu - \lambda - \frac{1}{2}} e^{-t}
\left( 1 + \frac{t}{z} \right)^{\mu + \lambda - \frac{1}{2}} dt
\quad \text{for $\mu - \lambda > -\frac{1}{2}$.}
\end{align}
From (\ref{density-bilateral-conv}) and (\ref{repr-whittaker}) we obtain for $x > 0$
\begin{align}\label{repr-dens-whittaker}
f(x) &= \frac{(\lambda^+)^{\alpha^+} (\lambda^-)^{\alpha^-}}{(\lambda^+ + \lambda^-)^{\frac{1}{2}(\alpha^+ + \alpha^-)}
\Gamma(\alpha^+)} x^{\frac{1}{2}(\alpha^+ + \alpha^-) - 1}
e^{-\frac{x}{2}(\lambda^+ - \lambda^-)}
\\ \notag & \quad \times W_{\frac{1}{2}(\alpha^+ - \alpha^-), \frac{1}{2}(\alpha^+ + \alpha^- - 1)}
(x(\lambda^+ + \lambda^-)).
\end{align}
By \cite[p. 1014]{Gradstein}, we can express the Whittaker function $W_{\lambda,\mu}(z)$ by
the Whittaker functions $M_{\lambda,\mu}(z)$, namely it holds
\begin{align*}
W_{\lambda,\mu}(z) = \frac{\Gamma(-2 \mu)}{\Gamma(\frac{1}{2} - \mu - \lambda)} M_{\lambda,\mu}(z)
+ \frac{\Gamma(2 \mu)}{\Gamma(\frac{1}{2} + \mu - \lambda)} M_{\lambda,-\mu}(z).
\end{align*}
For the Whittaker function $M_{\lambda,\mu}(z)$ the identity
\cite[p. 1014]{Gradstein}
\begin{align*}
M_{\lambda,\mu}(z) &= z^{\mu + \frac{1}{2}} e^{-\frac{z}{2}}
\Phi(\mu - \lambda + {\textstyle \frac{1}{2}}, 2 \mu + 1; z)
\end{align*}
is valid, with $\Phi(\alpha,\gamma;z)$ denoting the
\textit{confluent hypergeometric function} \cite[p. 1013]{Gradstein}
\begin{align}\label{def-confl-hyp-func}
\Phi(\alpha,\gamma;z) = 1 + \frac{\alpha}{\gamma} \frac{z}{1!} + \frac{\alpha (\alpha+1)}{\gamma (\gamma+1)}
\frac{z^2}{2!} + \frac{\alpha(\alpha+1)(\alpha+2)}{\gamma(\gamma+1)(\gamma+2)} \frac{z^3}{3!} + \ldots
\end{align}
Because of the series representation (\ref{def-confl-hyp-func}) of $\Phi(\alpha,\gamma;z)$, we can use
(\ref{repr-dens-whittaker}) in order to obtain density plots with a computer program, which is done in Section \ref{sec-shapes}.

If one of $\alpha^+, \alpha^-$ is an integer, the representation
becomes more convenient at one half of the real axis.

\begin{proposition}\label{prop-alpha-int}
Assume $\alpha^+ \in \mathbb{N} = \{ 1,2,\ldots \}$. Then it holds for each $x \in
(0,\infty)$
\begin{align*}
f(x) =
\frac{(\lambda^+)^{\alpha^+}(\lambda^-)^{\alpha^-}}{(\lambda^+ +
\lambda^-)^{\alpha^-} (\alpha^+ - 1)!} \left( \sum_{k=0}^{\alpha^+
- 1} a_k x^k \right) e^{-\lambda^+ x},
\end{align*}
where the coefficients $(a_k)_{k=0,\ldots,\alpha^+ - 1}$ are
given by
\begin{align*}
a_k = { \alpha^+ - 1 \choose k } \frac{1}{(\lambda^+ +
\lambda^-)^{\alpha^+ - 1 - k}} \prod_{l=0}^{\alpha^+ - 2 - k}
(\alpha^- + l).
\end{align*}
\end{proposition}

\begin{proof}
Since $\alpha^+$ is an integer, we can compute the integral appearing in (\ref{density-bilateral-conv}) by using the binomial expansion formula. The calculations are obvious.
\end{proof}

The symmetry relation (\ref{symm-rel}) and the identity \cite[p. 1017]{Gradstein}
\begin{align*}
W_{0,\mu}(z) = \sqrt{\frac{z}{\pi}} K_{\mu} \left( \frac{z}{2} \right),
\end{align*}
where $K_{\mu}(z)$ denotes the Bessel function of the third kind, imply that in the case
$\alpha^+ = \alpha^- =: \alpha$ the density (\ref{repr-dens-whittaker}) is of the form
\begin{align}\label{density-bilateral-vg}
f(x) = \frac{1}{\Gamma(\alpha)} \left( \frac{\lambda^+ \lambda^-}{\lambda^+ + \lambda^-} \right)^{\alpha}
|x|^{\alpha - 1} e^{- \frac{x}{2}(\lambda^+ + \lambda^-)} \sqrt{\frac{|x|(\lambda^+ + \lambda^-)}{\pi}}
K_{\alpha - \frac{1}{2}} \left( \frac{|x|}{2} (\lambda^+ + \lambda^-) \right)
\end{align}
for $x \in \mathbb{R} \setminus \{ 0 \}$. The density of a Variance Gamma distribution $VG(\mu,\sigma^2,\nu)$ is, according to \cite[Sec. 6.1.5]{Madan}, given by
\begin{align}\label{density-vg}
h(x) = \frac{2 \exp(\frac{\mu x}{\sigma^2})}{\nu^{1 / \nu} \sqrt{2 \pi} \sigma \Gamma(\frac{1}{\nu})}
\left( \frac{x^2}{\frac{2 \sigma^2}{\nu} + \mu^2} \right)^{\frac{1}{2\nu} - \frac{1}{4}}
K_{\frac{1}{\nu} - \frac{1}{2}} \left( \frac{1}{\sigma^2} \sqrt{x^2 \left( \frac{2 \sigma^2}{\nu} + \mu^2 \right)} \right).
\end{align}
Inserting the parametrization
\begin{align}\label{var-gamma-para}
(\mu,\sigma^2,\nu) := \left( \frac{\alpha}{\lambda^+} -
\frac{\alpha}{\lambda^-}, \frac{2 \alpha} {\lambda^+ \lambda^-},
\frac{1}{\alpha} \right)
\end{align}
into (\ref{density-vg}), we obtain the density (\ref{density-bilateral-vg}), showing that bilateral Gamma distributions with $\alpha^+ = \alpha^- =: \alpha$ are Variance Gamma with parameters given by (\ref{var-gamma-para}).

Conversely, for a bilateral Gamma distribution which is Variance Gamma it necessarily holds $\alpha^+ = \alpha^-$, see \cite[Thm. 3.3]{Kuechler-Tappe}.

\section{Smoothness}\label{sec-smoothness}

As we have pointed out in Section \ref{sec-basic}, bilateral Gamma distributions are selfdecomposable. Therefore, we may use the results of \cite{Sato-Yamazato} and \cite{Sato-Y} in the sequel.

The smoothness of the density depends on the parameters $\alpha^+$
and $\alpha^-$. Let $N \in \mathbb{N}_0 = \{ 0, 1, 2, \ldots \}$ be
the unique nonnegative integer satisfying $N < \alpha^+ + \alpha^-
\leq N+1$.

\begin{theorem}\label{thm-smoothness}
It holds $f \in C^N(\mathbb{R} \setminus \{0\})$ and $f \in C^{N-1}(\mathbb{R}) \setminus C^N(\mathbb{R})$.
\end{theorem}

\begin{proof}
This is an immediate consequence of \cite[Thm. 1.2]{Sato-Yamazato}.
\end{proof}

Thus, the $N$-th order
derivative of the density $f$ is not continuous. The only point of
discontinuity is zero. In Section \ref{sec-asymptotic}, we will explore the behaviour of
$f^{(N)}$ near zero.

The densities of bilateral Gamma distributions satisfy the following integro-differential equation.

\begin{proposition}
$f$ satisfies for $x \in \mathbb{R} \setminus \{ 0 \}$
\begin{align*}
x f'(x) &= (\alpha^+ + \alpha^- - 1) f(x) - \alpha^+ \lambda^+
\int_0^{\infty} f(x-u) e^{-\lambda^+ u} du
\\ &\quad - \alpha^- \lambda^- \int_0^{\infty} f(x+u) e^{-\lambda^- u} du.
\end{align*}
\end{proposition}

\begin{proof}
The assertion follows from Cor. 2.1 of \cite{Sato-Yamazato}.
\end{proof}

\section{Unimodality}\label{sec-unimodality}

Bilateral Gamma distributions are \textit{strictly unimodal}, which is the content of the next result.

\begin{theorem}\label{thm-unimodal}
There exists a point $x_0 \in \mathbb{R}$ such that $f$ is strictly
increasing on $(-\infty,x_0)$ and strictly decreasing on
$(x_0,\infty)$.
\end{theorem}

\begin{proof}
The existence of the mode $x_0$ is a
direct consequence of \cite[Thm. 1.4]{Sato-Yamazato}, because
neither the distribution function of a bilateral Gamma distribution nor its reflection is of type ${\rm I}_4$ in the sense of \cite{Sato-Yamazato}.
\end{proof}

We emphasize that the \textit{mode} $x_0$ from Theorem \ref{thm-unimodal} can, in general, not be determined explicitly. However, we get the following result, which narrows the location of the mode.

\begin{proposition}\label{prop-unimodal}
If $\alpha^+, \alpha^- \leq 1$, then $x_0 = 0$. Presumed $\alpha^+ > 1$ and $\alpha^- \leq 1$, it holds $x_0 \in (0,
\frac{\alpha^+ - 1}{\lambda^+})$. In the case $\alpha^+, \alpha^- > 1$ we have $x_0 \in (- \frac{\alpha^-
- 1}{\lambda^-}, \frac{\alpha^+ - 1}{\lambda^+})$, and it holds
\begin{align*}
x_0 = 0 \quad &\text{if and only if} \quad \lambda^- \alpha^+ -
\lambda^+ \alpha^- = \lambda^- - \lambda^+,
\\ x_0 > 0 \quad &\text{if and only if} \quad \lambda^- \alpha^+ -
\lambda^+ \alpha^- > \lambda^- - \lambda^+,
\\ x_0 < 0 \quad &\text{if and only if} \quad \lambda^- \alpha^+ -
\lambda^+ \alpha^- < \lambda^- - \lambda^+.
\end{align*}
\end{proposition}

\begin{proof}
The first statement is a consequence of parts (viii) and (ix) of \cite[Thm.
1.3]{Sato-Yamazato}.

Since the mode of a $\Gamma(\alpha,\lambda)$-distribution with
$\alpha > 1$ is given by $\frac{\alpha - 1}{\lambda}$, parts (ii)
and (iii) of \cite[Thm. 4.1]{Sato-Yamazato} yield the second assertion.

In the case $\alpha^+, \alpha^- > 1$, part (iv) of \cite[Thm. 4.1]{Sato-Yamazato} shows that $x_0 \in (- \frac{\alpha^- - 1}{\lambda^-}, \frac{\alpha^+ - 1}{\lambda^+})$. According to
Theorem \ref{thm-smoothness}, the density $f$ is continuously
differentiable. Using the representation (\ref{density-bilateral-conv}) and Lebesgue's
dominated convergence theorem, we obtain the first derivative for
$x \in (0,\infty)$
\begin{align*}
f'(x) &= \frac{(\lambda^+)^{\alpha^+}
(\lambda^-)^{\alpha^-}}{(\lambda^+ + \lambda^-)^{\alpha^-}
\Gamma(\alpha^+) \Gamma(\alpha^-)} \bigg[ -\lambda^+ e^{-\lambda^+
x} \int_0^\infty v^{\alpha^- - 1} \left( x + \frac{v}{\lambda^+ +
\lambda^-} \right)^{\alpha^+ - 1} e^{-v} dv
\\ &\quad \quad +(\alpha^+ - 1) e^{-\lambda^+ x}
\int_0^\infty v^{\alpha^- - 1} \left( x + \frac{v}{\lambda^+ +
\lambda^-} \right)^{\alpha^+ - 2} e^{-v} dv \bigg].
\end{align*}
Applying Lebesgue's dominated convergence theorem again, by the
continuity of $f'$ and the fact $\Gamma(x+1) = x \Gamma(x)$, $x
> 0$ we get
\begin{align*}
f'(0) = \frac{(\lambda^+)^{\alpha^+}
(\lambda^-)^{\alpha^-}}{(\lambda^+ + \lambda^-)^{\alpha^+ +
\alpha^- - 2}} \frac{\Gamma(\alpha^+ + \alpha^- -
2)}{\Gamma(\alpha^+ - 1) \Gamma(\alpha^-)} \left[ 1 -
\frac{\lambda^+}{\lambda^+ + \lambda^-} \cdot \frac{\alpha^+ + \alpha^- - 2}{\alpha^+ -
1} \right],
\end{align*}
which yields the remaining statement of the proposition.
\end{proof}

A particular consequence of Proposition \ref{prop-unimodal} is that for $\lambda^+ \gg \alpha^+$ and $\lambda^- \gg \alpha^-$ the mode $x_0$ is necessarily close to zero.

\section{Asymptotic behaviour}\label{sec-asymptotic}

We have seen in Section \ref{sec-smoothness} that for $N := \lceil \alpha^+ + \alpha^-\rceil - 1$ the $N$-th order
derivative of the density $f$ is not continuous. The only point of
discontinuity is zero. We will now explore the behaviour of
$f^{(N)}$ near zero. For the proof of the upcoming result, Theorem \ref{thm-nearzero}, we need the following properties of the \textit{Exponential Integral} \cite[Chap. 5]{Abramowitz}
\begin{align*}
E_1(x) := \int_1^{\infty} \frac{e^{-xt}}{t} dt, \quad x > 0.
\end{align*}
The Exponential Integral has the series expansion
\begin{align}\label{series-ei}
E_1(x) = - \gamma - \ln x - \sum_{n=1}^{\infty} \frac{(-1)^n}{n \cdot n!} x^n,
\end{align}
where $\gamma$ denotes Euler's constant
\begin{align*}
\gamma = \lim_{n \rightarrow \infty} \left[ \sum_{k=1}^n
\frac{1}{k} - \ln(n) \right].
\end{align*}
The derivative of the Exponential Integral is given by
\begin{align}\label{der-exp-int}
\frac{\partial}{\partial x}E_1(x) = - \frac{e^{-x}}{x}.
\end{align}
Due to symmetry relation (\ref{symm-rel}) it is, concerning the behaviour of
$f^{(N)}$ near zero, sufficient to treat the case $x \downarrow 0$.

\begin{theorem}\label{thm-nearzero}
Let $N := \lceil \alpha^+ + \alpha^- \rceil - 1$.
\begin{enumerate}
\item $\lim_{x \downarrow 0} f^{(N)}(x)$ is finite if and only if $\alpha^+ \in \mathbb{N} = \{ 1,2,\ldots \}$.

\item If $\alpha^+ \notin \mathbb{N}$ and $\alpha^+ + \alpha^- \notin \mathbb{N}$, then $f^{(N)}(x) \sim
\frac{C_1}{x^{\alpha}}$ as $x \downarrow 0$ for constants $C_1 \neq 0, \alpha \in (0,1)$.

\item Let $\alpha^+ \notin \mathbb{N}$ be such that $\alpha^+ + \alpha^- \in \mathbb{N}$. Then
$f^{(N)}(x) \sim M(x)$ as $x \rightarrow 0$, where
$M$ is a slowly varying function as $x \rightarrow 0$ satisfying $\lim_{x \rightarrow 0} M(x) = \infty$. Moreover, it holds
$\lim_{x \downarrow 0} ( f^{(N)}(x) - f^{(N)}(-x) ) = C_2 \in \mathbb{R}$.
\end{enumerate}
\end{theorem}

The constants in Theorem \ref{thm-nearzero} are given by
\begin{align*}
\alpha &= N + 1 - \alpha^+ - \alpha^-,
\\ C_1 &= \frac{(\lambda^+)^{\alpha^+}
(\lambda^-)^{\alpha^-} \sin(\alpha^+ \pi)}{\Gamma(\alpha^+ +
\alpha^- - N) \sin((\alpha^+ + \alpha^-)\pi)},
\\ C_2 &= \frac{(\lambda^+)^{\alpha^+}
(\lambda^-)^{\alpha^-}}{2} \Big( (-1)^{N+1} \cos (\alpha^+ \pi) + \cos (\alpha^- \pi) \Big).
\end{align*}

\begin{proof}
For $\alpha^+ \in \mathbb{N}$ we conclude the finiteness of the limit in the first statement from \cite[Thm. 3]{Sato-Y}, since for each $\beta \in (0,1)$ (recall that the function $k$ was defined in (\ref{fkt-k-self}))
\begin{align*}
\lim_{u \downarrow 0} u^{\beta - 1}(\alpha^+ - k(u)) = \lim_{u
\downarrow 0} u^{\beta - 1} \alpha^+ (1 - e^{-\lambda u}) = 0.
\end{align*}
In order to prove the rest of the theorem, we evaluate expressions (1.8)--(1.10) in \cite{Sato-Yamazato}, and then we apply \cite[Thm. 1.7]{Sato-Yamazato}. The constant $c$ in \cite[eqn. (1.8)]{Sato-Yamazato} is in the present situation
\begin{align}\notag
c &= \exp \bigg( (\alpha^+ + \alpha^-) \int_0^1 \frac{e^{-u} -
1}{u} du + (\alpha^+ + \alpha^-) \int_1^{\infty}
\frac{e^{-u}}{u} du
\\ \label{const-c-sato} & \quad \quad \quad - \int_1^{\infty} \frac{\alpha^+
e^{-\lambda^+ u} + \alpha^- e^{-\lambda^- u}}{u} du \bigg).
\end{align}
The first integral appearing in (\ref{const-c-sato}) is by (\ref{der-exp-int}) and the series expansion (\ref{series-ei})
\begin{align*}
\int_0^1 \frac{e^{-u} - 1}{u} du = \lim_{x \downarrow 0} \Big[
- E_1(u) - \ln u \Big]_{u=x}^{u=1} = -E_1(1) - \gamma.
\end{align*}
Using (\ref{der-exp-int}) and the fact $\lim_{x \rightarrow \infty} E_1(x) = 0$, see \cite[Chap. 5]{Abramowitz}, for each constant $\lambda > 0$ the identity
\begin{align*}
\int_1^{\infty} \frac{e^{-\lambda u}}{u} du = \lim_{x
\rightarrow \infty} \Big[ - E_1(\lambda u) \Big]_{u=1}^{u=x} =
E_1(\lambda)
\end{align*}
is valid. Thus, we obtain
\begin{align}\label{repr-const-c}
c = e^{-(\alpha^+ + \alpha^-) \gamma - \alpha^+ E_1(\lambda^+) -
\alpha^- E_1(\lambda^-)}.
\end{align}
The function $K(x)$ in \cite[eqn. (1.9)]{Sato-Yamazato} is in the present situation
\begin{align*}
K(x) = \exp \left( \int_{\abs{x}}^1 \frac{\alpha^+ + \alpha^-
-\alpha^+ e^{-\lambda^+ u} - \alpha^- e^{-\lambda^- u}}{u} du
\right).
\end{align*}
Since by (\ref{der-exp-int})
\begin{align*}
\int_{\abs{x}}^1 \frac{e^{-\lambda u}}{u} du =
E_1(\lambda \abs{x}) - E_1(\lambda) \quad \text{for $\lambda >
0$,}
\end{align*}
we obtain
\begin{align}\label{repr-k}
K(x) = e^{ \alpha^+ E_1(\lambda^+) + \alpha^- E_1(\lambda^-)}
\abs{x}^{-(\alpha^+ + \alpha^-)} e^{-\alpha^+ E_1(\lambda^+
\abs{x}) - \alpha^- E_1(\lambda^- \abs{x})}.
\end{align}
Using the series expansion (\ref{series-ei}), we get
\begin{align}\label{repr-func-k}
\lim_{x \rightarrow 0}K(x) = (\lambda^+)^{\alpha^+}
(\lambda^-)^{\alpha^-} e^{(\alpha^+ + \alpha^-) \gamma + \alpha^+
E_1(\lambda^+) + \alpha^- E_1(\lambda^-)},
\end{align}
showing that for the slowly varying function
\begin{align*}
L(x) = \int_{|x|}^1 \frac{K(u)}{u} du
\end{align*}
in \cite[eqn. (1.10)]{Sato-Yamazato} it holds
\begin{align}\label{lim-slow-var}
\lim_{x \rightarrow 0} L(x) = \infty.
\end{align}
Applying \cite[Thm. 1.7]{Sato-Yamazato} and relations (\ref{repr-const-c})--(\ref{lim-slow-var}) completes the proof.
\end{proof}

The asymptotic behaviour of the Whittaker function for large values of $|z|$ is, according to \cite[p. 1016]{Gradstein},
\begin{align*}
W_{\lambda, \mu}(z) \sim e^{-\frac{z}{2}} z^{\lambda} H(z)
\end{align*}
with $H$ denoting the function
\begin{align*}
H(z) = 1 + \sum_{k=1}^{\infty} \frac{\left[ \mu^2 - (\lambda - \frac{1}{2})^2 \right]
\left[ \mu^2 - (\lambda - \frac{3}{2})^2 \right] \cdots
\left[ \mu^2 - (\lambda - k + \frac{1}{2})^2 \right]}{k! z^k}.
\end{align*}
Obviously, it holds $H(z) \sim 1$ for $z \rightarrow \infty$.
Taking into account (\ref{symm-rel}) and (\ref{repr-dens-whittaker}), for $x \rightarrow \pm \infty$ the density has the asymptotic behaviour
\begin{align*}
f(x) &\sim C_3 x^{\alpha^+ - 1} e^{-\lambda^+ x} \quad \text{as $x
\rightarrow \infty$,}
\\ f(x) &\sim C_4 |x|^{\alpha^- - 1} e^{-\lambda^- |x|} \quad \text{as $x \rightarrow -\infty$,}
\end{align*}
where the constants $C_3, C_4 > 0$ are given by
\begin{align*}
C_3 = \frac{(\lambda^+)^{\alpha^+} (\lambda^-)^{\alpha^-}}{(\lambda^+
+ \lambda^-)^{\alpha^-} \Gamma(\alpha^+)} \quad \text{and} \quad C_4 =
\frac{(\lambda^+)^{\alpha^+} (\lambda^-)^{\alpha^-}}{(\lambda^+ +
\lambda^-)^{\alpha^+} \Gamma(\alpha^-)}.
\end{align*}
As a consequence, we obtain for the logarithmic density function $\ln f$
\begin{align*}
\lim_{x \rightarrow \infty} \frac{\ln f(x)}{x} = - \lambda^+ \quad \text{and} \quad
\lim_{x \rightarrow -\infty} \frac{\ln f(x)}{x} = \lambda^-.
\end{align*}
In particular, the density of a bilateral Gamma distribution is \textit{semiheavy tailed}.

\section{Shapes of the densities}\label{sec-shapes}

\begin{figure}[t]
   \centering
   \includegraphics[height=30ex,width=30ex]{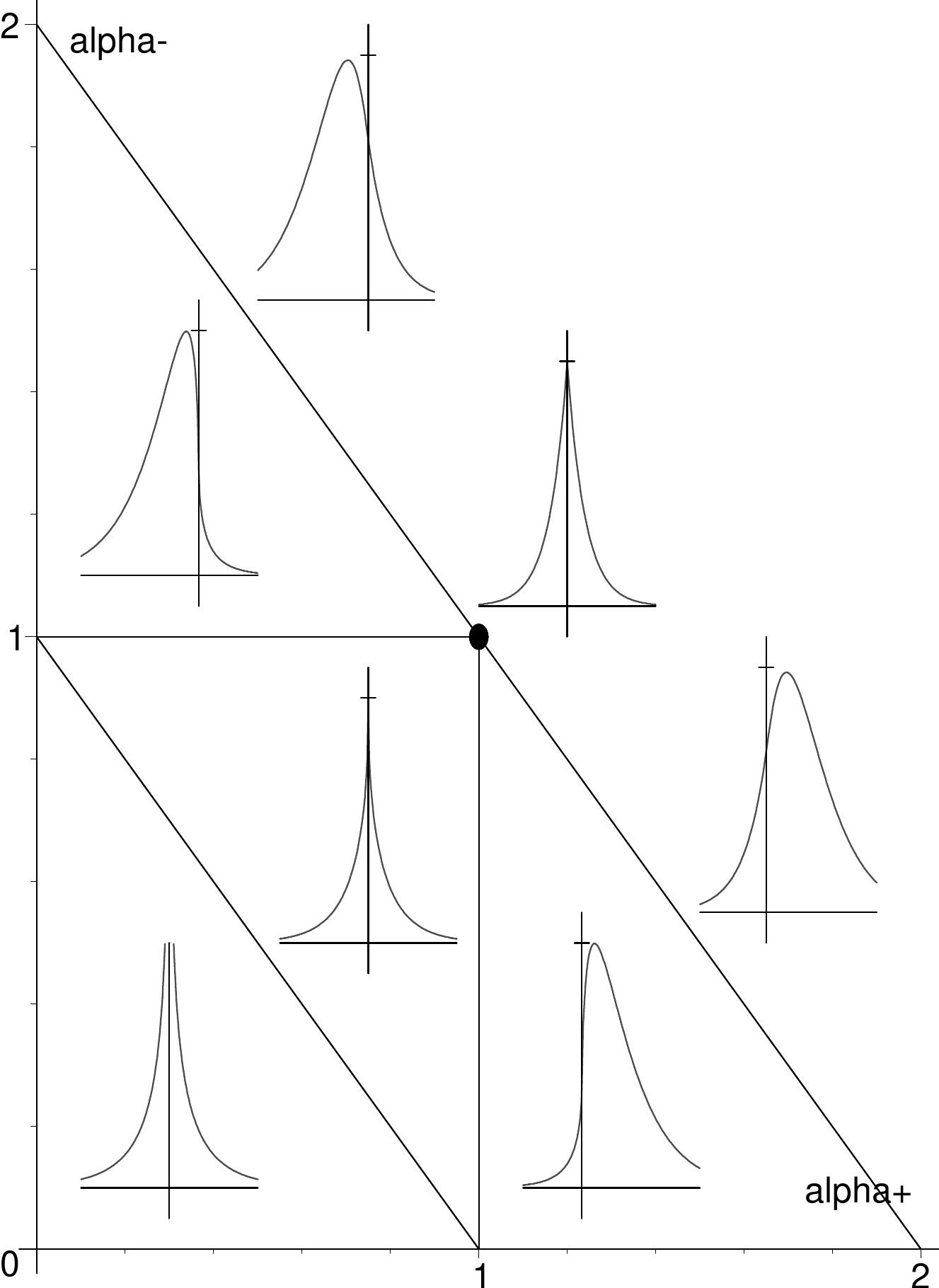}
   \caption{The shapes of $f$ for $\lambda^+ = \lambda^-$. Different choices of $\lambda^+$ and $\lambda^-$ may shift the mode and change the skewness.}
   \label{fig-shapes}
\end{figure}

The shapes of bilateral Gamma distributions can have remarkable differences. Using the results of the previous sections, we characterize typical shapes of their densities.

If $\alpha^+ + \alpha^- \leq 1$, then $f$ is not continuous at zero by Theorem \ref{thm-smoothness}. According to Theorem \ref{thm-nearzero}, it holds
\begin{align}\label{close-to-zero-obs}
\lim_{x \uparrow 0} f(x) = \infty \quad \text{and} \quad \lim_{x \downarrow 0} f(x) = \infty.
\end{align}
We infer that the density has a \textit{pole} at the mode $x_0 = 0$. Notice that in the special case $\alpha^+ + \alpha^- = 1$ the difference $f(x) - f(-x)$ tends to a finite value as $x \downarrow 0$ by the third statement of Theorem \ref{thm-nearzero}. We observe that densities with (\ref{close-to-zero-obs}) are appropriate for fitting data sets with many observations accumulating closely around zero.

If $1 < \alpha^+ + \alpha^- \leq 2$, then, by Theorem \ref{thm-smoothness}, $f$ is continuous on $\mathbb{R}$, but its derivative is not continuous at zero. Let us have a closer look at the behaviour of $f'$ near zero in this case.

\begin{itemize}
\item If $\alpha^+, \alpha^- \in (0,1)$ and $\alpha^+ + \alpha^- \in (1,2)$, it holds, according to Theorem \ref{thm-nearzero},
\begin{align*}
\lim_{x \uparrow 0} f'(x) = \infty \quad \text{and} \quad \lim_{x \downarrow 0} f'(x) = -\infty.
\end{align*}
Hence, we have a \textit{steep} mode of the density at zero with exploding slope from the left and from the right. This shape is also suitable for sets of data with many observations being close to zero.

\item If $\alpha^- < 1 < \alpha^+$, applying Theorem \ref{thm-nearzero} yields
\begin{align*}
\lim_{x \uparrow 0} f'(x) = \infty \quad \text{and} \quad \lim_{x \downarrow 0} f'(x) = \infty.
\end{align*}
Hence, the mode $x_0$ is located at the positive half axis and $f$ has infinite slope at zero. We remark that in the special case $\alpha^+ + \alpha^- = 2$ the difference $f'(x) - f'(-x)$ tends to a finite value as $x \downarrow 0$ by the third statement of Theorem \ref{thm-nearzero}.

\item If $\alpha^+, \alpha^- = 1$, we have a two-sided exponential distribution, which is in particular Variance Gamma, as we have shown at the end of Section \ref{sec-dens}. We obtain
\begin{align*}
\lim_{x \uparrow 0} f'(x) = C^- \quad \text{and} \quad \lim_{x \downarrow 0} f'(x) = -C^+
\end{align*}
with finite constants $C^-, C^+ \in (0,\infty)$. Consequently, we have a \textit{peak} mode of the density at zero with finite slope from both sides.

\end{itemize}

If $\alpha^+ + \alpha^- > 2$, then the density is smooth, that is at least of type $C^1(\mathbb{R})$ by Theorem \ref{thm-smoothness}.
Choosing $\lambda^+ \gg \alpha^+$ and $\lambda^- \gg \alpha^-$, the mode $x_0$ is necessarily close to zero by Proposition \ref{prop-unimodal}. Such shapes are in particular applicable for observations of financial data. We refer to \cite[Sec. 9]{Kuechler-Tappe}, where for a specific data set of stock returns the maximum likelihood estimation $\alpha^+ = 1.55$, $\lambda^+ = 133.96$, $\alpha^- = 0.94$, $\lambda^- = 88.92$ provided a good fit.

Summarizing the preceding results, Figure \ref{fig-shapes} provides an overview about typical shapes of bilateral Gamma densities.

\end{document}